\documentclass{amsart}

\usepackage{amsmath,amssymb,mathtools,euscript,tikz,units}
\usepackage{xcolor}
\usepackage{verbatim}
\usepackage{colonequals}
\usepackage[titletoc]{appendix}
\usepackage{yhmath}
\usepackage{graphicx}
\usepackage{enumerate}
\usepackage{enumitem}
\usepackage{url}
\usepackage[colorlinks,citecolor=blue,linkcolor=blue,urlcolor=blue]{hyperref}
\usepackage{fancyhdr}
\usepackage{lastpage}

\hypersetup{colorlinks=true, linkcolor=blue, citecolor=blue, urlcolor=blue}
\graphicspath{{./figures/}{./}}

\newcommand{\R}{\mathbb{R}}

\newcommand{\Rp}{\mathbb{R}_{+}}
\newcommand{\Rpo}{\mathbb{R}_{+}^{\circ}}

\newcommand{\eps}{\varepsilon}
\newcommand{\dd}{\,d}

\newcommand{\norm}[1]{\lVert #1\rVert}

\newcommand{\widehatc}[1]{\widehat{#1}_{c}}

\newcommand{\valpha}{\boldsymbol{\alpha}}

\DeclareMathOperator{\CornerCap}{Cap}

\theoremstyle{plain}
\newtheorem{theorem}{Theorem}[section]
\newtheorem{corollary}[theorem]{Corollary}
\newtheorem{proposition}[theorem]{Proposition}
\newtheorem{lemma}[theorem]{Lemma}

\theoremstyle{definition}

\newtheorem{question}[theorem]{Question}

\theoremstyle{remark}
\newtheorem*{remark}{Remark}

\usepackage[nameinlink]{cleveref}
\crefname{theorem}{Theorem}{Theorems}
\crefname{corollary}{Corollary}{Corollaries}
\crefname{proposition}{Proposition}{Propositions}
\crefname{lemma}{Lemma}{Lemmas}
\crefname{remark}{Remark}{Remarks}
\crefname{question}{Question}{Questions}
\Crefname{theorem}{Theorem}{Theorems}
\Crefname{corollary}{Corollary}{Corollaries}
\Crefname{proposition}{Proposition}{Propositions}
\Crefname{lemma}{Lemma}{Lemmas}
\Crefname{remark}{Remark}{Remarks}
\Crefname{question}{Question}{Questions}

\numberwithin{equation}{section}
\begin{document}

\title{The Ekeland--Nirenberg Variational Problem:\\
A Sharp Positivity Threshold and Extensions}

\author{Qi Guo}
\author{Xueping Huang}
\author{Yi Huang}

\begin{abstract}
We study the Ekeland--Nirenberg variational problem in the two-dimensional diagonal family
\[
 J_{a,c,d}(u)=\int_{\Rp^2}\bigl(u_{xy}^2+a u_x^2+c u_y^2+d u^2\bigr)\dd x\dd y,
 \qquad a,c,d>0,
\]
under the constraint $u(0,0)=1$. If $u_{a,c,d}$ is the unique minimizer and $K_{a,c,d}$ is its cosine kernel, we prove the sharp classification
\[
 K_{a,c,d}>0 \hbox{ on } \Rp^2\quad\Longleftrightarrow\quad
 u_{a,c,d}>0 \hbox{ on } \Rp^2\quad\Longleftrightarrow\quad d\le ac .
\]
Thus every supercritical triple $d>ac$ produces sign change. We also prove local sign-change stability under small two-dimensional non-diagonal perturbations and a sharp product-type $n$-dimensional diagonal threshold. The domain and evolution results are stated in precise auxiliary settings: a free-boundary capacity formulation for domains and a selected decaying branch of the second-order evolution equation. 
\end{abstract}

\maketitle


\tableofcontents

\section{Introduction}

We begin with the two-dimensional diagonal problem.  Let
\[
E_2=\{u:\Rp^2\to\R: u,u_x,u_y,u_{xy}\in L^2(\Rp^2)\}
\]
and, for $a,c,d>0$, set
\begin{equation*} 
J_{a,c,d}(u)=\int_{\Rp^2}\bigl(u_{xy}^2+a u_x^2+c u_y^2+d u^2\bigr)\dd x\dd y.
\end{equation*}
The constrained problem is
\[
\inf\{J_{a,c,d}(u):u\in E_2,\ u(0,0)=1\}.
\]
Ekeland and Nirenberg proved, in their general framework, that this problem has a unique minimizer and that the minimizer is smooth up to the boundary of the quadrant \cite{EkelandNirenberg2005}.  The remaining question was its sign.  At the end of their paper they ask whether the minimizer is positive.  This leads to the following natural conjecture.

\medskip
\noindent\textbf{Original positivity conjecture.}  For every $a,c,d>0$, the unique minimizer of $J_{a,c,d}$ under $u(0,0)=1$ is positive on $\Rp^2$.
\medskip

The conjecture has a concrete origin in mathematical finance.  Bouchard, Ekeland, and Touzi use Malliavin integration by parts to rewrite conditional expectations such as
\[
v(x)=\mathbb E[g(X_2)\mid X_1=x]
\]
as unconditional expectations with localizing weights \cite{BouchardEkelandTouzi2004}.  The localizer is not unique.  Different localizers give the same conditional expectation, but their Monte Carlo variances are different.  The analytic reduction in \cite{BouchardEkelandTouzi2004,EkelandNirenberg2005} leads to a quadratic minimization problem of the form above.  In this interpretation, $u(x,y)$ is the normalized localizing profile, $u(0,0)=1$ is the normalization that preserves the conditional expectation, and $J_{a,c,d}(u)$ is the part of the integrated mean-square error depending on the localizer.  Positivity of $u$ means that the optimal localizer is a genuine nonnegative averaging weight.  A negative part means that the variance optimum uses signed cancellation, closer to a control variate than to a probability weight.  Conditional expectations of this type appear in American and Bermudan option algorithms \cite{LongstaffSchwartz2001,Glasserman2004,BouchardWarin2012} and in Monte Carlo schemes for BSDEs \cite{BouchardTouzi2004,GobetLemorWarin2005,CrisanManolarakisTouzi2010,GobetTurkedjiev2016}.

For $(x,y)\in\Rp^2$ define the cosine kernel
\begin{equation}\label{eq:kernel-intro}
K_{a,c,d}(x,y)=\frac{4}{\pi^2}\int_0^\infty\int_0^\infty
\frac{\cos(x\xi)\cos(y\eta)}{\xi^2\eta^2+a\xi^2+c\eta^2+d}\dd\xi\dd\eta.
\end{equation}
We prove in Proposition \ref{prop:minimizer-kernel} that
\[
u_{a,c,d}(x,y)=\frac{K_{a,c,d}(x,y)}{K_{a,c,d}(0,0)}.
\]
Thus the positivity problem is exactly a kernel positivity problem.

\begin{theorem}[Sharp two-dimensional positivity threshold]\label{thm:main}
Let $a,c,d>0$.  Let $u_{a,c,d}$ be the unique minimizer of $J_{a,c,d}$ under $u(0,0)=1$, and let $K_{a,c,d}$ be the kernel in \eqref{eq:kernel-intro}.  Then the following are equivalent:
\begin{enumerate}[label=\textup{(\roman*)}]
\item $u_{a,c,d}(x,y)>0$ for every $(x,y)\in \Rp^2$;
\item $K_{a,c,d}(x,y)>0$ for every $(x,y)\in \Rp^2$;
\item $d\le ac$.
\end{enumerate}
Equivalently, every supercritical triple $d>ac$ produces sign change.
\end{theorem}

\begin{figure}[h]
\centering
\includegraphics[width=0.62\textwidth]{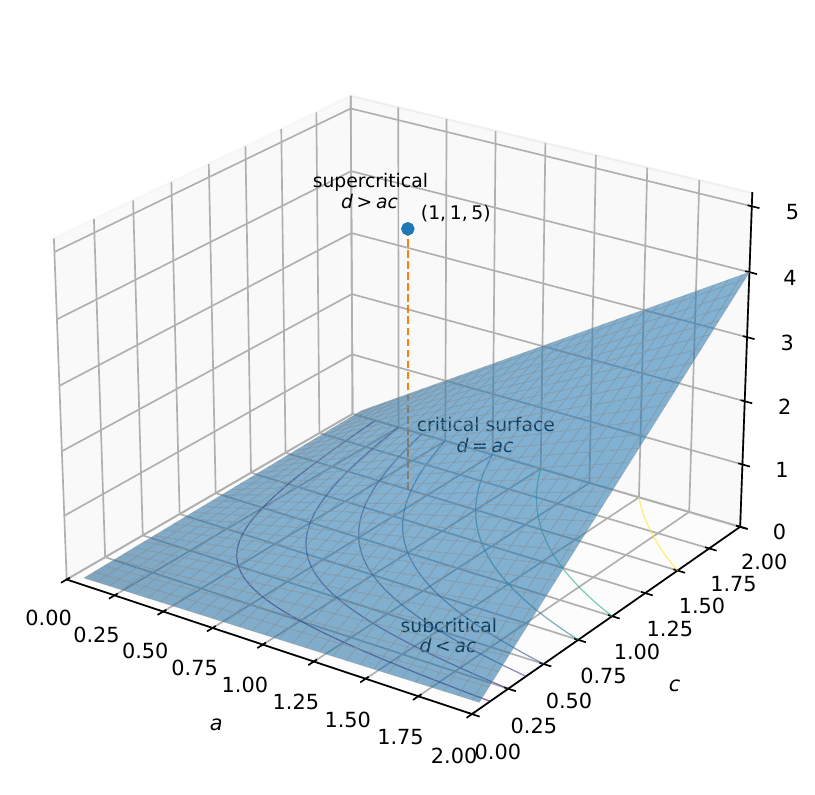}
\caption{The threshold surface $d=ac$ in the positive $(a,c,d)$-octant.  Points below or on the surface are positive.  Points above the surface are supercritical and produce sign change.  The point $(1,1,5)$ is a concrete counterexample.}
\label{fig:threshold}
\end{figure}

\begin{remark}
In the critical case $d=ac$, the kernel factorizes:
\[
K_{a,c,ac}(x,y)=\frac{1}{\sqrt{ac}}e^{-\sqrt c x-\sqrt a y},
\qquad
u_{a,c,ac}(x,y)=e^{-\sqrt c x-\sqrt a y}.
\]
In particular, $K_{a,c,ac}(0,0)=1/\sqrt{ac}$. Thus the separable exponential is exactly the threshold case.  It is the analytic counterpart of the separable exponential localizers in the variance-reduction problem.
\end{remark}

Theorem \ref{thm:main} has two immediate extensions.  First, every supercritical diagonal example is stable under small two-dimensional non-diagonal perturbations.  Second, the same result extends to a product-type diagonal subfamily of the original $n$-dimensional Ekeland--Nirenberg problem.

\begin{corollary}[Two-dimensional non-diagonal sign stability]\label{cor:intro-b}
Fix $a,c,d>0$ with $d>ac$.  For
\[
J_{a,b,c,d}(u)=\int_{\Rp^2}\bigl(u_{xy}^2+a u_x^2+2b u_xu_y+c u_y^2+d u^2\bigr)\dd x\dd y,
\]
there exists $b_0>0$ such that, for every $|b|<b_0$, the constrained Ekeland--Nirenberg minimizer under $u(0,0)=1$ changes sign.
\end{corollary}

This shows that the sign-changing counterexamples are not an artifact of the diagonal restriction $b=0$.  The conclusion is local in the coefficient $b$ and should not be read as a sharp global positivity criterion for the full non-diagonal two-dimensional problem.

Now we recall the original $n$-dimensional formulation using the notation fixed in Section~\ref{subsec:notation}.  The Ekeland--Nirenberg functional has the form
\[
J_Q(u)=\int_{\Rp^n}\left((D_{[n]}u)^2+
\bigl(Q\widetilde D^n_{n-1}u,\widetilde D^n_{n-1}u\bigr)\right)\dd\mathbf{x},
\]
where $Q$ is a symmetric positive definite matrix acting on the list $\widetilde D^n_{n-1}u$ of lower mixed derivatives. The following result is a sharp theorem for a product-type diagonal choice of $Q$.

\begin{corollary}[Product-type $n$-dimensional diagonal family]\label{cor:intro-nd}
Let $n\ge2$, let $\valpha=(\alpha_1,\ldots,\alpha_n)\in(0,\infty)^n$, and put $A_{\valpha}=\prod_{j=1}^n\alpha_j$.  Let $Q_{\valpha,d}$ be the diagonal matrix in the Ekeland--Nirenberg functional whose coefficient of $|D_Iu|^2$ is
\[
q_I=\prod_{j\notin I}\alpha_j
\quad\text{for }\emptyset\ne I\subsetneq[n],
\qquad
q_\emptyset=d.
\]
Equivalently,
\[
J^{(n)}_{\valpha,d}(u)=\int_{\Rp^n}\left(
(D_{[n]}u)^2+
\sum_{\emptyset\ne I\subsetneq[n]}\Bigl(\prod_{j\notin I}\alpha_j\Bigr)|D_Iu|^2+d|u|^2
\right)\dd\mathbf{x}.
\]
The constrained minimizer under $u(\mathbf{0})=1$ is positive on $\Rp^n$ if and only if
\[
d\le A_{\valpha}.
\]
If $d>A_{\valpha}$, then the minimizer changes sign.  This sign change remains true after adding any fixed finite family of sufficiently small non-diagonal cross terms
\[
2\theta_r\int_{\Rp^n}D_{S_r}uD_{T_r}u\dd\mathbf{x},
\qquad S_r\ne T_r\subsetneq [n].
\]
\end{corollary}

Long's ICVAM problem list records three questions of Ekeland about this unusual variational problem \cite[Section 3.2.2]{Long2008}.  The first is the positivity question, answered by Theorem \ref{thm:main} in the diagonal two-dimensional family.  The second asks what happens when the quadrant is replaced by other domains.  We use the natural relative corner capacity associated with the energy form.  This is the standard capacity viewpoint for point constraints and traces in Sobolev-type spaces; see, for example, \cite{Mazya2011,AdamsHedberg1996}.  For a domain $\Omega$ with corner $0$, define the free-boundary energy
\[
J_\Omega(u)=\int_\Omega\bigl(u_{xy}^2+a u_x^2+2b u_xu_y+c u_y^2+d u^2\bigr)\dd x\dd y,
\]
where $a>0$, $c>0$, $d>0$, and $ac-b^2>0$.  We define the free-boundary test class \(\mathcal D_{\rm fb}(\Omega)\) as follows.
A function \(v\) belongs to \(\mathcal D_{\rm fb}(\Omega)\) if
\[
v\in C^\infty(\Omega),\qquad
v,\ v_x,\ v_y,\ v_{xy}\in L^2(\Omega),
\]
its support is compact relative to \(\overline\Omega\), and \(v\) has a continuous
representative on \(\Omega\cup\{0\}\) in a relative neighborhood of the corner. The value \(v(0,0)\)
always means the value of this prescribed continuous representative, see Section \ref{sec:long-answers} for the detail.
Let $E(\Omega)$ be the completion of $\mathcal D_{\rm fb}(\Omega)$ under the norm induced by $J_\Omega^{1/2}$. Define the relative corner capacity
\begin{equation*} 
\CornerCap_\Omega(0)=\inf\{J_\Omega(v):v\in \mathcal D_{\rm fb}(\Omega),\ v(0,0)=1\}.
\end{equation*}

\begin{theorem}[Domain capacity criterion]\label{thm:intro-region-answer}
The constrained problem on $\Omega$ with corner $0$ 
\[\inf \{J_\Omega (u): u\in E(\Omega),\ u(0,0)=1\}\]
 has a unique minimizer if and only if
\[
\CornerCap_\Omega(0)>0.
\]

\end{theorem}

\begin{remark}
If $\CornerCap_\Omega(0)=0$, the  infimum is $0$; the corner condition is not a closed constraint in the energy completion, so no attained constrained minimizer is represented in $E(\Omega)$.
\end{remark}

\begin{remark}
For the cone
\[
C_{\theta_1,\theta_2}=\{r(\cos\theta,\sin\theta):r>0,\ \theta_1<\theta<\theta_2\},
\]
one has
\[
0\le\theta_1<\theta_2\le\frac\pi2,
\qquad
(\theta_1,\theta_2)\ne\left(0,\frac\pi2\right)
\quad\Longrightarrow\quad
\CornerCap_{C_{\theta_1,\theta_2}}(0)=0.
\]
The full quadrant has positive capacity.  If $\Omega$ contains the full quadrant with the same corner and the admissible restriction to the quadrant is well defined, then $\CornerCap_\Omega(0)>0$, but regularity is not automatic.  Ekeland and Nirenberg already point to the half-plane mechanism \cite[Section 2]{EkelandNirenberg2005}: in $H=\R\times\Rp$, for critical diagonal parameters $d=ac$, the minimizer is
\[
u_H(x,y)=e^{-\sqrt c |x|-\sqrt a y},
\]
which is not $C^1$ on the interior line $x=0$, $y>0$.
\end{remark}

\begin{remark}The capacity of a cone is not determined only by its opening
angle, since the energy is tied to the fixed coordinate directions. If a cone contains
one of the coordinate quadrants, then its corner capacity is positive by restriction
to that quadrant. On the other hand, many obtuse cones which do not contain a
coordinate quadrant have zero corner capacity; the same one-variable logarithmic
cut-off argument applies whenever the horizontal or vertical sections shrink
linearly at the vertex.
\end{remark}
The third Ekeland question asks for an evolution-equation analogue of the same phenomenon: good energy well-posedness without the smoothing one would expect from elliptic or parabolic equations.  The following theorem treats the selected decaying branch, equivalently a first-order nonlocal semigroup; it is not a theorem for arbitrary Cauchy data of the full second-order equation.

\begin{theorem}[Decaying-branch evolution analogue]\label{thm:intro-evolution-answer}
The equation
\[
u_{xxtt}-a u_{xx}-c u_{tt}+d u=0,
\qquad x\in\R,
\quad t>0,
\]
has a decaying branch which is equivalent to
\[
\partial_tu+T_{a,c,d}u=0,
\qquad
\widehat{T_{a,c,d}f}(\xi)
=\sqrt{\frac{a\xi^2+d}{\xi^2+c}}\,\widehat f(\xi).
\]
This flow is well posed and exponentially decaying in every Sobolev space $H^s(\R)$, but it has no Sobolev smoothing.  In the critical case $d=ac$, the data $f(x)=e^{-\sqrt c |x|}$ give
\[
u(t,x)=e^{-\sqrt c |x|-\sqrt a t},
\]
which remains non-$C^1$ for every $t>0$.
\end{theorem}

\textbf{Main idea.}  Since the Euler--Lagrange equation has no useful maximum principle, the sign is determined by the transform kernel.  The first step is to identify the minimizer as the normalized Riesz representative of the corner value; after the product cosine transform this representative has reciprocal symbol $W_{a,c,d}^{-1}$.  Below the threshold $d\le ac$, this reciprocal symbol has a positive Bessel--Laplace representation.  Applying the Gaussian cosine transform turns it into a positive integral formula for $K_{a,c,d}$, so positivity is immediate.  Above the threshold $d>ac$, the positive representation is lost.  We reduce by scaling to $K_m$, continue the one-dimensional formula into the upper half-plane, and cut along the natural segment $[i,i\sqrt m]$.  The jump across this cut gives an exact real integral.  Its large-$x$ contribution comes from the endpoint $i$, and the correct boundary-layer scale is $y\asymp x^{-1/2}$.  On this scale the kernel converges uniformly to a universal profile $H$.  Since $H(0)>0$ but $H(\pi)<0$, every supercritical kernel has a negative value.  The product-type $n$-dimensional theorem uses the same dichotomy: a positive $n$-parameter Laplace representation below the product threshold, and a two-dimensional frequency-face obstruction above it.

In section 2, we introduce notation and basic lemmas.  Section 3 shows the Bessel--Laplace representations and the subcritical positivity result.  Section 4 proves the supercritical sign-change theorem.  Sections 5 and 6 prove the two-dimensional non-diagonal stability and the $n$-dimensional extensions.  Section 7 gives the capacity and decaying-branch evolution results in their precise settings.  Section 8 records further questions and discussions.

\section{Preliminaries and technical tools}

This section collects the facts used later.  We first fix notation, then discuss the two-dimensional case, and finally record the basic tools that enter the proof.

\subsection{Notation, function spaces, and variational identities}\label{subsec:notation}

We use one convention for the half-line throughout the paper:
\[
 \Rp=[0,\infty),\qquad \Rpo=(0,\infty).
\]
Integrals over $\Rp^n$ mean Lebesgue integrals over $(\Rpo)^n$; changing the domain by boundary sets of measure zero does not change any integral. Pointwise positivity on $\Rp^n$ always refers to the continuous representative supplied by Lemma~\ref{lem:mixed-embedding}. Thus the boundary axes are included in positivity statements.
The following notation guide collects symbols that will be used later:
\begin{itemize}[leftmargin=1.7em]
\item $L^2(\Omega)$ is the Hilbert space of square-integrable functions on $\Omega$, and $H^s(\R)$ is the Sobolev space of order $s$.
\item $C_0(\Rp^n)$ denotes continuous functions on $\Rp^n$ that vanish at infinity; $E_n\hookrightarrow C_0(\Rp^n)$ means continuous embedding.
\item $\widehatc{u}$ denotes the product cosine transform of $u$; hats without the subscript $c$ are ordinary Fourier transforms.
\item $D_Iu=\prod_{j\in I}\partial_{x_j}u$ is a mixed partial derivative indexed by a subset $I\subset[n]$, with $D_\emptyset u=u$.
\item $\widetilde D^n_{n-1}u$ is the list of all derivatives $D_Iu$ with $I\subsetneq[n]$; it excludes only the highest mixed derivative $D_{[n]}u$.
\item $f^-:=\max\{-f,0\}$ is the negative part of a real function; $a.e.$ means almost everywhere.
\item $f\prec g$ means $f\le Cg$ with a uniform constant $C$, and $f\asymp g$ means both $f\prec g$ and $g\prec f$.
\item $J\Subset U$ means that the closure of $J$ is compact and contained in $U$.
\item $O(\cdot)$ and $o(\cdot)$ have their standard asymptotic meanings; all implicit constants are independent of the asymptotic variable under discussion unless dependence is explicitly indicated. $Q>0$ for a matrix means symmetric positive definite.
\item $\CornerCap_\Omega(0)$ denotes the relative corner capacity associated with the energy on a domain $\Omega$.
\end{itemize}

Vectors are written in bold when helpful, for instance $\mathbf{x}=(x_1,\ldots,x_n)$ and $\boldsymbol{\xi}=(\xi_1,\ldots,\xi_n)$. Parameter vectors such as $\valpha=(\alpha_1,\ldots,\alpha_n)$ are also bold. We write $[n]=\{1,\ldots,n\}$ and
\[
 A_n=\{\sigma=(\sigma_1,\ldots,\sigma_n):\sigma_j\in\{0,1\}\}.
\]
Each $\sigma\in A_n$ is identified with the subset $I(\sigma)=\{j:\sigma_j=1\}\subset[n]$. To avoid using two symbols for the same derivative, all formulas below use the subset notation
\[
 D_Iu=\prod_{j\in I}\partial_{x_j}u,
 \qquad D_\emptyset u=u.
\]
The highest mixed derivative is $D_{[n]}u$. The list of lower derivatives in the Ekeland--Nirenberg functional is
\[
 \widetilde D^n_{n-1}u=\{D_Iu:I\subset[n],\ I\ne[n]\}.
\]
The mixed Sobolev space is
\[ 
E_n=\{u:\Rp^n\to\R:D_Iu\in L^2(\Rp^n)\text{ for every }I\subset[n]\},\]
\[
\norm{u}_{E_n}^2=\sum_{I\subset[n]}\norm{D_Iu}_{L^2(\Rp^n)}^2.
\]
In particular,
\[
E_2=\{u:\Rp^2\to\R:u,u_x,u_y,u_{xy}\in L^2(\Rp^2)\}.
\]

\begin{lemma}[Mixed Sobolev embedding and point evaluation]\label{lem:mixed-embedding}
For every $n\ge1$, the product cosine transform identifies $E_n$ with the weighted transform space
\[
\int_{\Rp^n}\prod_{j=1}^n(1+\xi_j^2)|\widehatc{u}(\boldsymbol{\xi})|^2\dd\boldsymbol{\xi}<\infty .
\]
Moreover $E_n\hookrightarrow C_0(\Rp^n)$ and, for every $\mathbf p\in\Rp^n$,
\begin{equation}\label{eq:point-evaluation-bound}
|u(\mathbf p)|\le \norm{u}_{E_n}.
\end{equation}
In particular, the corner functional $u\mapsto u(\mathbf{0})$ is continuous on $E_n$.
\end{lemma}

\begin{proof}
For the tensor-product smooth core $C_c^\infty([0,\infty)^n)$, the one-dimensional cosine identity applied successively in each variable gives
\[
\norm{u}_{E_n}^2=
\int_{\Rp^n}\prod_{j=1}^n(1+\xi_j^2)|\widehatc{u}(\boldsymbol{\xi})|^2\dd\boldsymbol{\xi},
\]
and the identity extends to $E_n$ by completion. The weak mixed derivatives are handled by applying the one-dimensional identity variable by variable, which is precisely the product-cosine Plancherel identity for the factors $D_Iu$. The inverse product cosine formula gives
\[
u(\mathbf p)=\left(\frac2\pi\right)^{n/2}
\int_{\Rp^n}\widehatc{u}(\boldsymbol{\xi})
\prod_{j=1}^n\cos(p_j\xi_j)\dd\boldsymbol{\xi}.
\]
By Cauchy--Schwarz,
\[
|u(\mathbf p)|\le \left(\frac2\pi\right)^{n/2}
\left(\int_{\Rp^n}\prod_{j=1}^n(1+\xi_j^2)|\widehatc{u}|^2\dd\boldsymbol{\xi}\right)^{1/2}
\left(\int_{\Rp^n}\prod_{j=1}^n\frac{\dd\boldsymbol{\xi}}{1+\xi_j^2}\right)^{1/2}.
\]
Since $\int_0^\infty(1+\xi^2)^{-1}\dd\xi=\pi/2$, the constants cancel and \eqref{eq:point-evaluation-bound} follows. The same estimate also gives $\widehatc{u}\in L^1(\Rp^n)$. Hence the inverse product cosine transform has a continuous representative vanishing at infinity, by the Riemann--Lebesgue lemma applied to the even extension of $\widehatc{u}$.
\end{proof}
The embedding $E_n\hookrightarrow C_0(\Rp^n)$ will be used repeatedly in this explicit form.
For the two-dimensional diagonal problem we write
\[
J_{a,c,d}(u)=\int_{\Rp^2}\bigl(u_{xy}^2+a u_x^2+c u_y^2+d u^2\bigr)\dd x\dd y
\]
and denote its bilinear form by
\[
B_{a,c,d}(u,v)=\int_{\Rp^2}\bigl(u_{xy}v_{xy}+a u_xv_x+c u_yv_y+d uv\bigr)\dd x\dd y.
\]
For the non-diagonal two-dimensional form we use
\[
B_{a,b,c,d}(u,v)=\int_{\Rp^2}\bigl(u_{xy}v_{xy}+a u_xv_x+b(u_xv_y+u_yv_x)+c u_yv_y+d uv\bigr)\dd x\dd y.
\]
The diagonal symbol and kernel are
\[
W_{a,c,d}(\xi,\eta)=\xi^2\eta^2+a\xi^2+c\eta^2+d,
\quad 
K_{a,c,d}(x,y)=\frac{4}{\pi^2}\int_0^\infty\int_0^\infty
\frac{\cos(x\xi)\cos(y\eta)}{W_{a,c,d}(\xi,\eta)}\dd\xi\dd\eta.
\]
We shall also use the original Ekeland--Nirenberg $n$-dimensional functional
\[
J_Q(u)=\int_{\Rp^n}\left((D_{[n]}u)^2+
\bigl(Q\widetilde D^n_{n-1}u,\widetilde D^n_{n-1}u\bigr)\right)\dd\mathbf{x},
\]
where $Q$ is symmetric positive definite. Its associated bilinear form is denoted by $B_Q$.

We now record the first and second variation. Let $u_*$ be the minimizer of a positive quadratic functional $J(u)=B(u,u)$ under the constraint $u(0)=1$. The tangent space to the constraint is
\[
T=\{h:h(0)=0\}.
\]
For every $h\in T$,
\[
0=\left.\frac{d}{dt}\right|_{t=0}J(u_*+th)=2B(u_*,h),
\]
so
\begin{equation}\label{eq:first-variation-tangent}
B(u_*,h)=0\qquad (h(0)=0).
\end{equation}
Since $u_*(0)=1$, any test function $v$ can be decomposed as
\[
v=\bigl(v-v(0)u_*\bigr)+v(0)u_*,
\]
and the first term has value $0$ at the corner. Hence \eqref{eq:first-variation-tangent} is equivalent to
\begin{equation}\label{eq:first-variation-all}
B(u_*,v)=J(u_*)v(0)\qquad \text{for all }v.
\end{equation}
This is the source of the $v(0)$ term in the weak Euler--Lagrange identity. It is a Lagrange multiplier term coming from the corner constraint, not a boundary term produced by a free variation.

The second variation on the constraint is even simpler. If $h,k\in T$, then
\[
\delta^2J(u_*)[h,k]=2B(h,k),
\qquad
J(u_*+h)=J(u_*)+J(h).
\]
Thus the constrained Hessian has no kernel on $T$: if $B(h,k)=0$ for every $k\in T$, then taking $k=h$ gives $B(h,h)=0$, hence $h=0$ by positive definiteness. This is the nondegeneracy used implicitly in the uniqueness and stability arguments below.

In the two-dimensional non-diagonal case, the weak identity \eqref{eq:first-variation-all} reads
\begin{equation}\label{eq:weak-nondiagonal}
B_{a,b,c,d}(u,v)=J_{a,b,c,d}(u)v(0,0)\qquad (v\in E_2).
\end{equation}
If $u$ is smooth and decays sufficiently at infinity, integration by parts gives the formal Euler--Lagrange system
\[
u_{xxyy}-a u_{xx}-2b u_{xy}-c u_{yy}+d u=0\qquad \text{in }(\Rpo)^2,
\]
with natural boundary conditions
\[
u_{xyy}(0,y)-a u_x(0,y)-b u_y(0,y)=0\qquad (y>0),
\]
\[
u_{xxy}(x,0)-b u_x(x,0)-c u_y(x,0)=0\qquad (x>0),
\]
and corner conditions
\[
u(0,0)=1,
\qquad
u_{xy}(0,0)=J_{a,b,c,d}(u).
\]
This formal system is included only as a guide; the rigorous statement used below is the weak identity \eqref{eq:weak-nondiagonal}.

\subsection{Cosine transform and the two-dimensional kernel}
For a function $f$ on $\Rp$, use the normalized cosine transform
\[
(Cf)(\xi)=\sqrt{\frac{2}{\pi}}\int_0^\infty f(x)\cos(x\xi)\dd x,
\qquad \xi\ge 0.
\]
It is the Fourier transform of the even extension of $f$ restricted to $[0,\infty)$.

\begin{lemma}[One-dimensional cosine identities]\label{lem:cosine-one}
If $f\in H^1(\Rp)$ and $\lambda>0$, then
\[
\int_0^\infty |f(x)|^2\dd x=\int_0^\infty |Cf(\xi)|^2\dd\xi,
\]
and
\[
\int_0^\infty \bigl(|f'(x)|^2+\lambda |f(x)|^2\bigr)\dd x
=\int_0^\infty (\xi^2+\lambda)|Cf(\xi)|^2\dd\xi.
\]
\end{lemma}

\begin{remark}
This is the standard Plancherel identity for the Fourier cosine transform: the even extension of an $H^1(\Rp)$ function lies in $H^1(\R)$, and its weak derivative is the odd extension of $f'$ with no boundary Dirac mass; see \cite{ErdelyiTables1954}.
\end{remark}

For $u\in E_2$ write
\[
\widehatc{u}(\xi,\eta)=(C_xC_yu)(\xi,\eta),
\qquad \xi,\eta\ge 0.
\]
Set
\[
W_{a,c,d}(\xi,\eta)=\xi^2\eta^2+a\xi^2+c\eta^2+d.
\]

\begin{proposition}[Spectral form and corner evaluation]\label{prop:spectral}
For every $u\in E_2$,
\[
J_{a,c,d}(u)=\int_0^\infty\int_0^\infty W_{a,c,d}(\xi,\eta)|\widehatc{u}(\xi,\eta)|^2\dd\xi\dd\eta.
\]
Moreover,
\[
\int_0^\infty\int_0^\infty \frac{\dd\xi\dd\eta}{W_{a,c,d}(\xi,\eta)}<\infty,
\]
$\widehatc{u}\in L^1(\Rp^2)$, and
\[
u(0,0)=\frac{2}{\pi}\int_0^\infty\int_0^\infty \widehatc{u}(\xi,\eta)\dd\xi\dd\eta.
\]
\end{proposition}

\begin{proof}
The energy identity follows by applying Lemma \ref{lem:cosine-one} in each variable.  The mixed derivative gives the factor $\xi^2\eta^2$, the $x$-derivative gives $\xi^2$, the $y$-derivative gives $\eta^2$, and the $L^2$ term gives $1$.

For the reciprocal estimate, integrate first in $\eta$:
\[
\int_0^\infty\frac{\dd\eta}{(\xi^2+c)\eta^2+a\xi^2+d}
=\frac{\pi}{2\sqrt{\xi^2+c}\sqrt{a\xi^2+d}}.
\]
This is bounded near $\xi=0$ and is $O(\xi^{-2})$ as $\xi\to\infty$.  The outer integral therefore converges.

Cauchy--Schwarz gives
\[
\int |\widehatc{u}|\le
\left(\int W_{a,c,d}|\widehatc{u}|^2\right)^{1/2}
\left(\int W_{a,c,d}^{-1}\right)^{1/2}<\infty.
\]
The inverse product cosine transform may therefore be evaluated at $(0,0)$, giving the displayed formula for the corner value.
\end{proof}

\begin{proposition}[Kernel formula for the minimizer]\label{prop:minimizer-kernel}
The kernel $K_{a,c,d}$ in \eqref{eq:kernel-intro} is finite for every $(x,y)\in\Rp^2$, and $K_{a,c,d}(0,0)>0$.  The unique minimizer under $u(0,0)=1$ is
\[
u_{a,c,d}(x,y)=\frac{K_{a,c,d}(x,y)}{K_{a,c,d}(0,0)}.
\]
In particular, the sign of the minimizer is the sign of the kernel.
\end{proposition}

\begin{proof}
Absolute convergence follows from the reciprocal estimate in Proposition \ref{prop:spectral}.  If $u(0,0)=1$, then the corner formula and Cauchy--Schwarz give
\[
1\le \frac{2}{\pi}J_{a,c,d}(u)^{1/2}
\left(\int_0^\infty\int_0^\infty W_{a,c,d}^{-1}\dd\xi\dd\eta\right)^{1/2}.
\]
Equality holds exactly when $\widehatc{u}$ is a positive constant multiple of $W_{a,c,d}^{-1}$.  Inverting the product cosine transform gives a multiple of $K_{a,c,d}$.  The normalization $u(0,0)=1$ fixes the constant and gives the formula above.  Strict positivity of $K_{a,c,d}(0,0)$ is immediate from the positive integrand at the corner.
\end{proof}

\begin{proposition}[Scaling]\label{prop:scaling}
Let $K_m=K_{1,1,m}$.  For every $a,c,d>0$,
\[
K_{a,c,d}(x,y)=\frac1{\sqrt{ac}}K_{d/(ac)}(\sqrt c x,\sqrt a y).
\]
Thus the two-dimensional sign problem depends only on $m=d/(ac)$.
\end{proposition}

\begin{proof}
In the kernel integral set $\xi=\sqrt c\,r$ and $\eta=\sqrt a\,s$.  The denominator becomes
\[
ac(r^2s^2+r^2+s^2+d/(ac)),
\]
and the Jacobian is $\sqrt{ac}$.  This gives the formula.
\end{proof}

\begin{lemma}[A standard cosine integral]\label{lem:standard-cosine}
For $\Lambda>0$ and $y\ge0$,
\[
\int_0^\infty\frac{\cos(y\eta)}{\eta^2+\Lambda^2}\dd\eta=\frac{\pi}{2\Lambda}e^{-\Lambda y}.
\]
\end{lemma}

\begin{proof}
This is the classical residue computation for the Poisson kernel on the line; see \cite{ErdelyiTables1954}.
\end{proof}

\begin{proposition}[One-dimensional kernel formula]\label{prop:onedim}
For every $a,c,d>0$ and $(x,y)\in\Rp^2$,
\[
K_{a,c,d}(x,y)=\frac{2}{\pi}\int_0^\infty
\frac{\cos(x\xi)\exp\left(-y\sqrt{\frac{a\xi^2+d}{\xi^2+c}}\right)}
{\sqrt{\xi^2+c}\sqrt{a\xi^2+d}}\dd\xi.
\]
\end{proposition}

\begin{proof}
For fixed $\xi$ write
\[
\xi^2\eta^2+a\xi^2+c\eta^2+d=(\xi^2+c)\left(\eta^2+\frac{a\xi^2+d}{\xi^2+c}\right).
\]
Lemma \ref{lem:standard-cosine} evaluates the inner $\eta$-integral and gives the displayed formula.
\end{proof}

\subsection{Bessel and Laplace identities}
We use the Bessel functions
\[
J_0(z)=\sum_{k=0}^\infty\frac{(-1)^k(z^2/4)^k}{(k!)^2},
\qquad
I_0(z)=\sum_{k=0}^\infty\frac{(z^2/4)^k}{(k!)^2}.
\]

\begin{lemma}[Two Laplace identities]\label{lem:bessel-laplace}
Let $\alpha>0$, $v\ge0$, and $\lambda\ge0$.  Then
\[
\int_0^\infty e^{-\alpha u}I_0(2\sqrt{\lambda uv})\dd u=\frac1\alpha e^{\lambda v/\alpha},
\]
and
\[
\int_0^\infty e^{-\alpha u}J_0(2\sqrt{\lambda uv})\dd u=\frac1\alpha e^{-\lambda v/\alpha}.
\]
\end{lemma}

\begin{proof}
For $I_0$, insert the power series and use Tonelli's theorem:
\[
\int_0^\infty e^{-\alpha u}\sum_{k=0}^\infty\frac{(\lambda uv)^k}{(k!)^2}\dd u
=\sum_{k=0}^\infty\frac{(\lambda v)^k}{(k!)^2}\frac{k!}{\alpha^{k+1}}
=\frac1\alpha e^{\lambda v/\alpha}.
\]
For $J_0$, the same computation is first done for the partial sums.  The partial sums are dominated by $I_0(2\sqrt{\lambda uv})$, and the first identity makes this dominating function integrable in $u$.  Dominated convergence then gives the identity for $J_0$.
\end{proof}

\subsection{Analytic continuation and branch cuts}
The supercritical argument uses one standard contour fact.  If an integrand is analytic in an upper half-plane with a finite slit removed, decays on the large semicircle, and has integrable square-root singularities at the endpoints of the slit, then the real-line integral equals the jump integral across the slit.  We shall use this in Section 4 and verify all hypotheses directly there.

\begin{remark}
The branch cut in our problem is not artificial.  For $m>1$ the analytic continuation of
\[
\lambda_m(z)=\sqrt{\frac{z^2+m}{z^2+1}}
\]
has a natural cut on the imaginary segment $[i,i\sqrt m]$.  This cut is exactly where the supercritical sign oscillation is born.
\end{remark}

\section{Bessel--Laplace representations and the threshold}

The reciprocal symbol can be written as
\[
M_{a,c,d}(s,t)=\frac1{st+as+ct+d}
=\frac1{(s+c)(t+a)+(d-ac)}.
\]
This identity already displays the threshold $d=ac$.  The kernel is the product cosine transform of $M_{a,c,d}(\xi^2,\eta^2)$.

\begin{proposition}[Bessel--Laplace representations]\label{prop:bessel-rep}
Let $a,c,d>0$.

If $d\le ac$ and $\nu=ac-d$, then
\[
M_{a,c,d}(s,t)=\int_0^\infty\int_0^\infty e^{-su-tv}e^{-cu-av}I_0(2\sqrt{\nu uv})\dd u\dd v.
\]

If $d>ac$ and $\mu=d-ac$, then
\[
M_{a,c,d}(s,t)=\int_0^\infty\int_0^\infty e^{-su-tv}e^{-cu-av}J_0(2\sqrt{\mu uv})\dd u\dd v.
\]
\end{proposition}

\begin{proof}
We prove the first formula.  Put $\alpha=s+c$ and $\beta=t+a$.  Lemma \ref{lem:bessel-laplace} gives
\[
\int_0^\infty e^{-\alpha u}I_0(2\sqrt{\nu uv})\dd u=\frac1\alpha e^{\nu v/\alpha}.
\]
The remaining $v$-integral equals
\[
\frac1\alpha\int_0^\infty e^{-(\beta-\nu/\alpha)v}\dd v
=\frac1{\alpha\beta-\nu}
=\frac1{st+as+ct+d}.
\]
The supercritical formula is identical, using the $J_0$ identity instead.
\end{proof}

\begin{corollary}[Positive kernel formula below the threshold]\label{cor:positive-laplace}
If $d\le ac$ and $\nu=ac-d$, then
\[
K_{a,c,d}(x,y)=\frac1\pi\int_0^\infty\int_0^\infty
\frac{e^{-cu-av-x^2/(4u)-y^2/(4v)}}{\sqrt{uv}}
I_0(2\sqrt{\nu uv})\dd u\dd v.
\]
In particular, $K_{a,c,d}(x,y)>0$ for every $(x,y)\in\Rp^2$.
\end{corollary}

\begin{proof}
Insert the subcritical representation of $M_{a,c,d}$ into the kernel formula and use the Gaussian cosine integral
\[
\int_0^\infty e^{-u\xi^2}\cos(x\xi)\dd\xi=\frac{\sqrt\pi}{2\sqrt u}e^{-x^2/(4u)}.
\]
All factors in the final integral are positive.  The interchange of integrals follows by truncating in $(u,v)$ and then using dominated convergence.  A simple domination is obtained from $I_0(z)\le e^z$ and
\[
2\sqrt{\nu uv}\le \sqrt{\frac{\nu}{ac}}(cu+av),
\]
where $\sqrt{\nu/(ac)}<1$ because $d>0$.
\end{proof}

\begin{proposition}[Subcritical positivity]\label{prop:subcrit-pos}
If $d\le ac$, then the minimizer $u_{a,c,d}$ is strictly positive on $\Rp^2$.
\end{proposition}

\begin{proof}
By Proposition \ref{prop:minimizer-kernel}, we have $u_{a,c,d}=K_{a,c,d}/K_{a,c,d}(0,0)$ with positive denominator, it follows from 
Corollary \ref{cor:positive-laplace}, i.e. $K_{a,c,d}>0$ that $u_{a,c,d}>0$.
\end{proof}

A smooth function $M$ on $(0,\infty)^2$ is called jointly completely monotone if
\[
(-1)^{p+q}\partial_s^p\partial_t^q M(s,t)\ge0
\qquad(p,q\ge0,\ s,t>0).
\]

\begin{proposition}[Complete monotonicity diagnosis]\label{prop:complete-monotonicity}
For $a,c,d>0$, the following are equivalent:
\begin{enumerate}[label=\textup{(\alph*)}]
\item $d\le ac$;
\item $M_{a,c,d}$ is jointly completely monotone on $\Rp^2$;
\item $M_{a,c,d}$ has a positive two-parameter Laplace representation.
\end{enumerate}
\end{proposition}

\begin{proof}
The implication (a)$\Rightarrow$(c) is Proposition \ref{prop:bessel-rep}.  The implication (c)$\Rightarrow$(b) follows by differentiating under the positive Laplace integral.

It remains to prove (b)$\Rightarrow$(a).  Suppose $d>ac$ and put $\mu=d-ac>0$.  Then
\[
M_{a,c,d}(s,t)=\frac1{(s+c)(t+a)+\mu}.
\]
A direct induction gives, for every integer $m\ge1$,
\[
(-1)^{m+1}\partial_s^m\partial_t M_{a,c,d}(s,t)
=\frac{m!(t+a)^{m-1}\bigl((s+c)(t+a)-m\mu\bigr)}{((s+c)(t+a)+\mu)^{m+2}}.
\]
At $(0,0)$ the numerator contains $ac-m\mu$.  Choosing $m>ac/\mu$ makes it negative, contradicting complete monotonicity.
\end{proof}

\begin{remark}
The complete monotonicity result is not needed to force sign change.  It explains the same threshold from a transform viewpoint: below the threshold the reciprocal symbol is generated by a positive measure, while above the threshold this positive-measure structure fails at a finite mixed derivative.
\end{remark}

\begin{proposition}[Axis positivity]\label{prop:axis}
For every $a,c,d>0$,
\[
K_{a,c,d}(x,0)>0\quad (x\ge0),
\qquad
K_{a,c,d}(0,y)>0\quad (y\ge0).
\]
More precisely,
\[
K_{a,c,d}(x,0)=\frac1{\pi^{3/2}\sqrt a}\int_0^\infty\int_0^\infty
\frac{e^{-cu-(d/a)v-x^2/(4(u+v))}}{\sqrt{uv(u+v)}}\dd u\dd v,
\]
and
\[
K_{a,c,d}(0,y)=\frac1{\pi^{3/2}\sqrt c}\int_0^\infty\int_0^\infty
\frac{e^{-au-(d/c)v-y^2/(4(u+v))}}{\sqrt{uv(u+v)}}\dd u\dd v.
\]
\end{proposition}

\begin{proof}
Set $y=0$ in Proposition \ref{prop:onedim}.  Use
\[
\frac1{\sqrt{\xi^2+c}}=\frac1{\sqrt\pi}\int_0^\infty u^{-1/2}e^{-u(\xi^2+c)}\dd u,
\]
and the analogous identity for $(a\xi^2+d)^{-1/2}$.  The remaining $\xi$-integral is the Gaussian cosine integral with parameter $u+v$.  Every factor is positive.
\end{proof}

\section{The supercritical proof}

By the scaling in Proposition \ref{prop:scaling}, it is enough to analyze
\[
K_m=K_{1,1,m},\qquad m>1.
\]
The proof has three parts: an exact branch-cut formula, a boundary-layer limit, and a sign test for the limiting profile.

\subsection{Analytic continuation and the branch-cut formula}
From Proposition \ref{prop:onedim}, we have 
\[
K_m(x,y)=\frac2\pi\int_0^\infty
\frac{\cos(x\xi)e^{-y\lambda_m(\xi)}}{(\xi^2+1)\lambda_m(\xi)}\dd\xi,
\qquad
\lambda_m(\xi)=\sqrt{\frac{\xi^2+m}{\xi^2+1}}.
\]

\begin{lemma}[The slit domain]\label{lem:slit}
Let
\[
D=\{z\in\mathbb C:\operatorname{Im}z>0\}\setminus [i,i\sqrt m].
\]
Then
\[
q_m(z)=\frac{z^2+m}{z^2+1}
\]
does not meet $(-\infty,0]$ on $D$.  Hence the principal square root $\lambda_m(z)=\sqrt{q_m(z)}$ is analytic on $D$ and has nonnegative real part there.
\end{lemma}

\begin{proof}
Write $z=s+it$ with $t>0$.  A direct computation gives
\[
\operatorname{Im}q_m(z)=\frac{2(1-m)st}{|z^2+1|^2}.
\]
Thus $q_m(z)$ can be real only when $s=0$.  On the imaginary axis, $z=it$ gives
\[
q_m(it)=\frac{m-t^2}{1-t^2},
\]
which is negative exactly for $1<t<\sqrt m$.  That interval is precisely the removed slit.  The statement follows from the principal branch of the square root.
\end{proof}

\begin{figure}[t]
\centering
\includegraphics[width=0.78\textwidth]{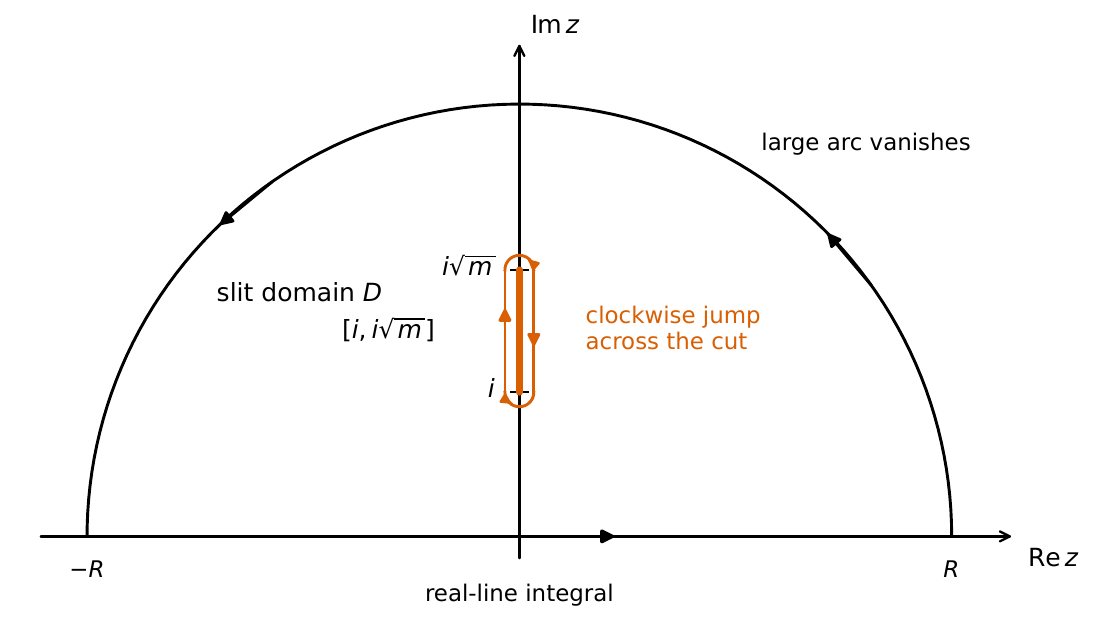}
\caption{The contour in the slit upper half-plane $D=\{\operatorname{Im}z>0\}\setminus [i,i\sqrt m]$.  The real-line integral is deformed to the boundary of the slit.  The large arc and the small endpoint detours vanish; the remaining contribution is the clockwise jump across $[i,i\sqrt m]$.}
\label{fig:branch-cut}
\end{figure}

\begin{proposition}[Exact branch-cut formula]\label{prop:branch}
Let $m>1$.  For every $x>0$ and $y\ge0$,
\[
K_m(x,y)=\frac2\pi\int_1^{\sqrt m}
\frac{e^{-xt}\cos\left(y\sqrt{\frac{m-t^2}{t^2-1}}\right)}{\sqrt{t^2-1}\sqrt{m-t^2}}\dd t.
\]
\end{proposition}

\begin{proof}
Since the integrand is even in $\xi$,
\[
K_m(x,y)=\frac1\pi\int_{-\infty}^{\infty}F(\xi)\,d\xi,
\qquad
F(z)=e^{ixz}\frac{e^{-y\lambda_m(z)}}{(z^2+1)\lambda_m(z)},
\]
where $\lambda_m$ is the principal square root on the slit domain of Lemma \ref{lem:slit}.  The branch is fixed by $\lambda_m(r)>0$ for $r>0$.

Let $\Gamma_{R,\rho}$ be the positively oriented boundary of the upper half-disc $|z|<R$, with the slit $[i,i\sqrt m]$ removed and with small semicircular detours of radius $\rho$ around the endpoints $i$ and $i\sqrt m$.  On the large arc, $\lambda_m(z)=1+O(|z|^{-2})$ and
\[
F(z)=O(|z|^{-2}e^{-x\operatorname{Im}z}),
\]
so the large-arc integral tends to zero as $R\to\infty$.  Near either endpoint of the slit, one factor has a square-root zero or pole and the remaining factors are bounded; hence $F(z)=O(|z-z_0|^{-1/2})$.  The endpoint detours therefore have integrals $O(\rho^{1/2})\to0$.

It remains to compute the contribution from the two sides of the slit.  Write $z=it$, $1<t<\sqrt m$, and
\[
\beta_m(t)=\sqrt{\frac{m-t^2}{t^2-1}}.
\]
Approaching the upward-oriented slit from the right side sends $q_m(it)$ to the negative real axis from below; approaching from the left side sends it from above.  Hence the principal square-root boundary values are
\[
\lambda_{m,+}(it)=-i\beta_m(t),
\qquad
\lambda_{m,-}(it)=i\beta_m(t),
\]
where $+$ and $-$ denote the right and left sides of the upward-oriented slit.  Since $(it)^2+1=-(t^2-1)$,
\[
F_+(it)-F_-(it)
=-2i\,e^{-xt}\cos(y\beta_m(t))\frac1{\sqrt{t^2-1}\sqrt{m-t^2}}.
\]
The boundary of the slit in the punctured upper half-plane is traversed clockwise.  Equivalently, after the real-line part is kept with its usual orientation, the slit contribution is
\[
\int_1^{\sqrt m}(F_+(it)-F_-(it))i\,dt.
\]
Cauchy's theorem, followed by $R\to\infty$ and $\rho\to0$, gives the displayed formula after division by $\pi$. The assumption $x>0$ is used in the large-arc estimate; axis positivity is handled separately in Proposition~\ref{prop:axis}.
\end{proof}

\subsection{Boundary layer and universal profile}
The formula in Proposition \ref{prop:branch} shows that large $x$ is governed by the endpoint $t=1$.  The natural scale is $t-1\asymp x^{-1}$, which corresponds to $y\asymp x^{-1/2}$.

Define
\[
H(\alpha)=\int_0^\infty e^{-u}u^{-1/2}\cos\left(\frac{\alpha}{\sqrt u}\right)\dd u,
\qquad \alpha\ge0.
\]

\begin{proposition}[Boundary-layer limit]\label{prop:boundary}
Let $m>1$ and $L>0$.  Then
\[
\sup_{0\le\rho\le L}\left|
e^x\sqrt x\,K_m\left(x,\frac{\rho}{\sqrt x}\right)
-\frac{2}{\pi\sqrt{2(m-1)}}H\left(\rho\sqrt{\frac{m-1}{2}}\right)
\right|\longrightarrow0
\]
as $x\to\infty$.
\end{proposition}

\begin{proof}
Fix $1<t_*<\sqrt m$ and split the branch-cut integral into the intervals $[1,t_*]$ and $[t_*,\sqrt m]$.
On $[t_*,\sqrt m]$ the function
\[
\frac{1}{\sqrt{t^2-1}\sqrt{m-t^2}}
\]
is integrable.  Hence, after multiplication by $e^x\sqrt x$, this part is bounded by
\[
C_m\sqrt x\,e^{-x(t_*-1)},
\]
uniformly for $0\le\rho\le L$.  It tends to zero.

On $[1,t_*]$ put $t=1+u/x$.  After extending the transformed integrand by zero for $u>x(t_*-1)$, this part becomes
\[
\int_0^\infty F_x(\rho,u)\,du,
\]
where
\[
F_x(\rho,u)=\frac2\pi e^{-u}
\frac{\cos\left(\frac{\rho}{\sqrt x}\sqrt{\frac{m-(1+u/x)^2}{(1+u/x)^2-1}}\right)}
{\sqrt{x}\sqrt{(1+u/x)^2-1}\sqrt{m-(1+u/x)^2}}
\]
for $0<u<x(t_*-1)$, and $F_x=0$ otherwise.  For fixed $u>0$,
\[
\sqrt{x}\sqrt{(1+u/x)^2-1}\sqrt{m-(1+u/x)^2}\longrightarrow \sqrt{2u(m-1)},
\]
and
\[
\frac1{\sqrt x}\sqrt{\frac{m-(1+u/x)^2}{(1+u/x)^2-1}}\longrightarrow \sqrt{\frac{m-1}{2u}}.
\]
Thus
\[
F_x(\rho,u)\longrightarrow
\frac2{\pi\sqrt{2(m-1)}}e^{-u}u^{-1/2}
\cos\left(\rho\sqrt{\frac{m-1}{2u}}\right).
\]
The convergence is uniform for $(\rho,u)$ in every compact rectangle $[0,L]\times[\delta,R]$ with $0<\delta<R<\infty$.

We also need a uniform majorant.  On $1\le t\le t_*$ we have
\[\sqrt{m-t^2}\ge c_m>0,\qquad \sqrt{t^2-1}\ge \sqrt{2(t-1)}.\]
After the change $t=1+u/x$, this gives
\[
|F_x(\rho,u)|\le C_m e^{-u}u^{-1/2},
\qquad 0\le\rho\le L,
\]
and the right-hand side is integrable on $(0,\infty)$.  Given $\eps>0$, choose $0<\delta<R<\infty$ so that the majorant has total mass less than $\eps$ on $(0,\delta)\cup(R,\infty)$.  On $[\delta,R]$ use the uniform convergence just proved.  This proves the desired uniform convergence in $\rho$.
\end{proof}

\begin{lemma}[A one-crossing rearrangement estimate]\label{lem:one-crossing}
Let $q$ and $h$ be increasing real functions on an interval $[A,B]$, with $h\ge0$. If $\int_A^B q(t)\,dt\ge0$ and $q$ has at most one zero, then
\[
\int_A^B q(t)h(t)\,dt\ge0.
\]
\end{lemma}

\begin{proof}
If $q$ has no zero, the claim is immediate. Otherwise let $\tau$ be its zero. Then
\[
\int_A^B qh=h(\tau)\int_A^B q+
\int_A^B q(t)\bigl(h(t)-h(\tau)\bigr)\,dt.
\]
The first term is nonnegative by hypothesis. In the second integral the two factors have the same sign on both sides of $\tau$, so it is also nonnegative.
\end{proof}

\begin{lemma}[The profile changes sign, with certified estimates]\label{lem:H-negative}
The function $H$ is continuous on $[0,\infty)$, real analytic on $(0,\infty)$, satisfies $H(0)=\sqrt\pi$, and
\[
H(\pi)<0.
\]
In fact, $H(\pi)<-3/200$.
\end{lemma}

\begin{proof}
Continuity follows from dominated convergence, since $e^{-u}u^{-1/2}$ is integrable.  The value $H(0)=\sqrt\pi$ is the Gamma integral.  For analyticity on compact subintervals of $(0,\infty)$, use $u=\alpha^2/t^2$ to write
\[
H(\alpha)=2\alpha\int_0^\infty t^{-2}e^{-\alpha^2/t^2}\cos t\,dt.
\]
All $\alpha$-derivatives are dominated by finite sums of $t^{-N}e^{-c/t^2}$ near $0$ and powers of $t^{-1}$ near infinity.

It remains to prove a negative value with certified estimates.  With $u=\pi^2/t^2$,
\[
H(\pi)=\int_0^\infty w(t)\cos t\,dt,
\qquad
w(t)=\frac{2\pi}{t^2}e^{-\pi^2/t^2}.
\]
The function $w$ is increasing on $(0,\pi)$ and decreasing on $(\pi,\infty)$.  The initial positive part satisfies
\[
\int_0^{\pi/2}w(t)\,dt=2\int_2^\infty e^{-s^2}\,ds\le \frac{e^{-4}}2<\frac1{100},
\]
because $\int_a^\infty e^{-s^2}\,ds\le e^{-a^2}/(2a)$ for $a>0$ and $e^4>50$.

For the alternating tail, set
\[
A_j=\int_{(2j+1)\pi/2}^{(2j+3)\pi/2}w(t)|\cos t|\,dt,
\qquad j\ge0.
\]
Since $w$ is decreasing on $[\pi,\infty)$, the sequence $A_1,A_2,\ldots$ is decreasing.  Hence
\[
\int_{\pi/2}^\infty w(t)\cos t\,dt
=-A_0+A_1-A_2+\cdots\le -(A_0-A_1).
\]
We prove $A_0-A_1>1/40$.  Write
\[
A_0-A_1=\int_{\pi/2}^{3\pi/2}\bigl(w(t)-w(t+\pi)\bigr)|\cos t|\,dt.
\]
On $[\pi/2,\pi]$, let $q(t)=w(t)-w(t+\pi)$ and $h(t)=|\cos t|$.  Both $q$ and $h$ are increasing.  Also
\[
\int_{\pi/2}^{\pi}q(t)\,dt
=2\int_1^2 e^{-s^2}\,ds-2\int_{1/2}^{2/3}e^{-s^2}\,ds>0.
\]
The last inequality is certified as follows.  Since $e^{-s^2}$ is convex on $[1,2]$, the composite midpoint rule with six equal subintervals gives a lower bound:
\[
\int_1^2 e^{-s^2}\,ds\ge
\frac16\sum_{j=0}^{5}\exp\!\left[-\left(\frac{13+2j}{12}\right)^2\right]
\ge \frac16\sum_{j=0}^{5}\left(1-\frac{(13+2j)^2}{80\cdot12^2}\right)^{80}>
\frac{25}{192}.
\]
Here the second inequality uses $(1-X/N)^N\le e^{-X}$ for $0<X<N$, and the final strict rational inequality is obtained by clearing denominators.  On the other hand,
\[
\int_{1/2}^{2/3}e^{-s^2}\,ds\le \frac16 e^{-1/4}<\frac{25}{192},
\]
because $e^{1/4}>1+1/4+1/32>32/25$.  Therefore $\int_{\pi/2}^{\pi}q>0$. Since $q$ and $h$ are increasing, Lemma~\ref{lem:one-crossing} gives
\[
\int_{\pi/2}^{\pi}q(t)h(t)\,dt\ge0.
\]

On $[\pi,3\pi/2]$,
\[
\int_{\pi}^{3\pi/2}\bigl(w(t)-w(t+\pi)\bigr)|\cos t|\,dt
\ge w(3\pi/2)-w(2\pi).
\]
Finally,
\[
w(3\pi/2)-w(2\pi)=\frac8{9\pi}e^{-4/9}-\frac1{2\pi}e^{-1/4}>\frac1{40},
\]
using $e^{-4/9}\ge 5/9$, $e^{-1/4}<25/32$, and $\pi<4$.  Thus the tail is $<-1/40$, while the initial positive part is $<1/100$, and hence
\[
H(\pi)<\frac1{100}-\frac1{40}=-\frac3{200}.
\]
\end{proof}

\begin{remark}
The exact constant is not important.  What matters is that the universal profile is positive at $0$ and negative at an explicit point.  This makes the supercritical sign change uniform in the parameter $m>1$ after scaling.
\end{remark}

\subsection{Proof of the sharp threshold and boundary-layer sign loss}

\begin{proof}[Proof of Theorem \ref{thm:main}]
Proposition \ref{prop:minimizer-kernel} gives the equivalence between the sign of the minimizer and the sign of the kernel.  Proposition \ref{prop:subcrit-pos} proves positivity when $d\le ac$.

Assume $d>ac$ and put $m=d/(ac)>1$.  Choose
\[
L_m=\pi\sqrt{\frac2{m-1}}.
\]
Then $L_m\sqrt{(m-1)/2}=\pi$.  Proposition \ref{prop:boundary} and Lemma \ref{lem:H-negative} give
\[
\lim_{x\to\infty} e^x\sqrt x\,K_m\left(x,\frac{L_m}{\sqrt x}\right)
=\frac2{\pi\sqrt{2(m-1)}}H(\pi)<0.
\]
Thus $K_m(x,L_m/\sqrt x)<0$ for all large $x$.  The scaling formula transfers this negative value back to $K_{a,c,d}$.  Hence supercritical kernels, and therefore supercritical minimizers, change sign.
\end{proof}

\begin{corollary}[Boundary-layer sign loss]\label{cor:first-scale}
Let $m>1$ and set $A_m=\sqrt{(m-1)/2}$.  If $J\Subset\{\alpha>0:H(\alpha)<0\}$, then there is $X_J>0$ such that
\[
\{(x,y):x>X_J,\ A_m\sqrt x\,y\in J\}\subset\{K_m<0\}.
\]
If $J\Subset\{\alpha\ge0:H(\alpha)>0\}$, the same domain is positive for all large $x$.  In particular,
\[
K_m\left(x,\frac{\pi}{A_m\sqrt x}\right)<0
\]
for all sufficiently large $x$.  The same conclusion holds with any point in a small interval around $\pi$.
\end{corollary}

\begin{proof}
This is just the uniform convergence in Proposition \ref{prop:boundary}.  On a compact subset of $\{H<0\}$, the limiting profile has a strictly negative maximum; on a compact subset of $\{H>0\}$, it has a strictly positive minimum.  Lemma \ref{lem:H-negative} supplies at least one negative interval, because $H(\pi)<0$ and $H$ is continuous.
\end{proof}

\begin{corollary}[A concrete example]\label{cor:concrete}
For $(a,c,d)=(1,1,5)$, the unique minimizer changes sign.  More precisely,
\[
K_{1,1,5}\left(X,\frac{\pi}{\sqrt{2X}}\right)<0
\]
for all sufficiently large $X$.
\end{corollary}

\begin{proof}
Here $m=5$, so $L_m=\pi/\sqrt2$.  The conclusion is the special case of the last part of the proof of Theorem \ref{thm:main}.
\end{proof}

\section{Two-dimensional non-diagonal stability}

This section proves Corollary \ref{cor:intro-b}.  The point is simple but important: once a diagonal minimizer is negative at one point, a small bounded perturbation of the quadratic form cannot remove that negative value.

Let
\[
A_b(u,v)=\int_{\Rp^2}\bigl(u_{xy}v_{xy}+a u_xv_x+b(u_xv_y+u_yv_x)+c u_yv_y+d uv\bigr)\dd x\dd y.
\]
The corresponding energy is $J_{a,b,c,d}(u)=A_b(u,u)$.

\begin{lemma}[Coercivity for small $b$]\label{lem:b-coercive}
For fixed $a,c,d>0$ there is $b_1>0$ such that $A_b$ is continuous, symmetric, and coercive on $E_2$ whenever $|b|<b_1$.
\end{lemma}

\begin{proof}
For $b=0$,
\[
A_0(u,u)\ge \gamma_0\norm{u}_{E_2}^2,
\qquad
\gamma_0=\min\{1,a,c,d\}>0.
\]
The mixed part satisfies
\[
\left|2\int u_xu_y\right|\le \norm{u_x}_2^2+\norm{u_y}_2^2\le \norm{u}_{E_2}^2.
\]
Thus $A_b(u,u)\ge(\gamma_0-|b|)\norm{u}_{E_2}^2$.  Take $b_1=\gamma_0/2$.
\end{proof}

Let $\ell(u)=u(0,0)$.  Since point evaluation is continuous on $E_2$, Lax--Milgram gives a unique $k_b\in E_2$ satisfying
\[
A_b(k_b,v)=\ell(v)\qquad (v\in E_2).
\]
As before, the constrained minimizer is
\[
u_b=\frac{k_b}{\ell(k_b)}.
\]

\begin{lemma}[Continuity in $b$]\label{lem:b-continuity}
As $b\to0$, $k_b\to k_0$ in $E_2$.  Hence $u_b(p)\to u_0(p)$ for every fixed point $p\in\Rp^2$.
\end{lemma}

\begin{proof}
Let $T_b:E_2\to E_2^*$ be the operator associated with $A_b$.  Then $T_b=T_0+bS$, where $S$ is bounded.  For $|b|<b_1$, the inverses $T_b^{-1}$ are uniformly bounded by coercivity.  The resolvent identity gives
\[
k_b-k_0=T_b^{-1}\ell-T_0^{-1}\ell=-bT_b^{-1}S k_0,
\]
which tends to $0$ in $E_2$.  Point evaluation is continuous, and $\ell(k_b)\to\ell(k_0)>0$, so the normalized minimizers converge pointwise.
\end{proof}

\begin{proof}[Proof of Corollary \ref{cor:intro-b}]
Since $d>ac$, Theorem \ref{thm:main} gives a point $p$ with $u_0(p)<0$.  By Lemma \ref{lem:b-continuity}, $u_b(p)<0$ for all sufficiently small $|b|$.  The constraint gives $u_b(0,0)=1>0$.  Hence $u_b$ changes sign.
\end{proof}

\begin{proposition}[A calibrating non-diagonal critical exponential]\label{prop:b-critical}
Assume $a,c>0$ and $|b|<\sqrt{ac}$.  Set
\[
\beta=\frac{b}{\sqrt{ac}},
\qquad r=\sqrt{1+\beta},
\qquad d_c=(\sqrt{ac}+b)^2.
\]
At $d=d_c$, the constrained minimizer of $J_{a,b,c,d}$ is
\[
u_c(x,y)=\exp\{-r(\sqrt c x+\sqrt a y)\}.
\]
\end{proposition}

\begin{proof}
Scale $X=\sqrt c x$, $Y=\sqrt a y$, and write $U(X,Y)=u(X/\sqrt c,Y/\sqrt a)$.  Up to the positive factor $\sqrt{ac}$, the energy becomes
\[
\int_{\Rp^2}\bigl(U_{XY}^2+U_X^2+2\beta U_XU_Y+U_Y^2+mU^2\bigr)\dd X\dd Y,
\qquad m=d/(ac).
\]
At $d=d_c$, $m=(1+\beta)^2=r^4$.  Let $\phi=e^{-r(X+Y)}$.  For every test function $V$,
\[
A_{\beta,r^4}(\phi,V)
=r^2\int_{\Rp^2}\partial_{XY}(\phi V)\dd X\dd Y
=r^2V(0,0).
\]
The last equality follows first for the half-space smooth core, for instance functions smooth up to the boundary with compact support in $[0,\infty)^2$, by integrating $\partial_{XY}(\phi V)$ over the rectangle containing the support. The identity then extends to all $V\in E_2$ because both $V\mapsto A_{\beta,r^4}(\phi,V)$ and $V\mapsto V(0,0)$ are continuous on $E_2$, and this smooth core is dense in $E_2$ by truncation, reflection, and mollification.  Thus $\phi/r^2$ is the Riesz representative of the corner functional.  After normalization by the corner value, the minimizer is $\phi$.  Returning to $(x,y)$ gives the formula.
\end{proof}

\begin{remark}
Proposition \ref{prop:b-critical} is a calibration result, not a sharp non-diagonal positivity theorem.  It shows that the explicit positive exponential lies on the curved surface $d=(\sqrt{ac}+b)^2$.  The full sharp threshold for arbitrary $b$ remains open.
\end{remark}

\section{The \texorpdfstring{$n$}{n}-dimensional Ekeland--Nirenberg setting}

This section proves Corollary \ref{cor:intro-nd}. We use the unified subset notation of Section~\ref{subsec:notation}: $D_{[n]}u$ is the highest mixed derivative and $\widetilde D^n_{n-1}u$ is the list of all lower derivatives. The general problem in \cite{EkelandNirenberg2005} uses
\[
J_Q(u)=\int_{\Rp^n}\left((D_{[n]}u)^2+
\bigl(Q\widetilde D^n_{n-1}u,\widetilde D^n_{n-1}u\bigr)\right)\dd\mathbf{x},
\]
with $Q$ symmetric positive definite. We prove a sharp sign theorem for the product-type diagonal subfamily and then prove perturbative stability for finite non-diagonal cross terms.

\subsection{The product-type diagonal subfamily}
Let $n\ge2$, let $\valpha=(\alpha_1,\ldots,\alpha_n)\in(0,\infty)^n$, and set
\[
A_{\valpha}=\prod_{j=1}^n\alpha_j.
\]
The product-type diagonal energy is
\[
J^{(n)}_{\valpha,d}(u)=\int_{\Rp^n}\left(
(D_{[n]}u)^2+
\sum_{\emptyset\ne I\subsetneq[n]}\Bigl(\prod_{j\notin I}\alpha_j\Bigr)|D_Iu|^2+d|u|^2\right)\dd\mathbf{x}.
\]
This is exactly a diagonal choice of $Q$ in the original formulation: the coefficient of $u^2$ is $d$, the coefficient of $|D_Iu|^2$ is $\prod_{j\notin I}\alpha_j$ for $\emptyset\ne I\subsetneq[n]$, and the coefficient of $(D_{[n]}u)^2$ is normalized to one.

Its product cosine symbol is
\[
W_{\valpha,d}(\boldsymbol{\xi})=\prod_{j=1}^n(\xi_j^2+\alpha_j)+d-A_{\valpha}.
\]
Indeed, expanding the product gives all mixed derivative terms, while the constant term $A_{\valpha}$ is replaced by $d$.  The associated kernel is
\[
K^{(n)}_{\valpha,d}(\mathbf{x})=\left(\frac2\pi\right)^n\int_{\Rp^n}
\frac{\prod_{j=1}^n\cos(x_j\xi_j)}{W_{\valpha,d}(\boldsymbol{\xi})}\,d\boldsymbol{\xi}.
\]
The same Cauchy--Schwarz argument as in Proposition \ref{prop:minimizer-kernel} gives
\[
u^{(n)}_{\valpha,d}(\mathbf{x})=\frac{K^{(n)}_{\valpha,d}(\mathbf{x})}{K^{(n)}_{\valpha,d}(\mathbf{0})}.
\]
\begin{lemma}[Integrability of the product-type reciprocal]\label{lem:nd-reciprocal-integrable}
For every $d>0$,
\[
W_{\valpha,d}^{-1}\in L^1(\Rp^n).
\]
More precisely,
\[
W_{\valpha,d}(\boldsymbol{\xi})\ge
\min\{1,d/A_{\valpha}\}\prod_{j=1}^n(\xi_j^2+\alpha_j).
\]
\end{lemma}

\begin{proof}
Let $P(\boldsymbol{\xi})=\prod_j(\xi_j^2+\alpha_j)$.  Since $P(\boldsymbol{\xi})\ge A_{\valpha}$,
\[
\frac{W_{\valpha,d}(\boldsymbol{\xi})}{P(\boldsymbol{\xi})}
=1+\frac{d-A_{\valpha}}{P(\boldsymbol{\xi})}
\ge \min\{1,d/A_{\valpha}\}.
\]
The reciprocal of $P$ is integrable as a product of one-dimensional integrable functions.

\end{proof}

\begin{lemma}[Half-line Abel approximation]\label{lem:half-line-abel}
Let $k\ge1$ and set
\[
\widetilde P_\eps(\xi)=\frac2\pi\frac{\eps}{\eps^2+\xi^2},\qquad \xi\ge0.
\]
If $f$ is bounded on $\Rp^k$ and continuous at the origin, then
\[
\int_{\Rp^k}f(\boldsymbol{\zeta})\prod_{j=1}^k\widetilde P_\eps(\zeta_j)\,d\boldsymbol{\zeta}
\longrightarrow f(\mathbf 0)
\qquad(\eps\downarrow0).
\]
\end{lemma}

\begin{proof}
The kernels are nonnegative and have mass one on the half-line. Given $\delta>0$, choose a small box $[0,r]^k$ on which $|f(\boldsymbol{\zeta})-f(\mathbf0)|<\delta$. The contribution of this box differs from $f(\mathbf0)$ by at most $\delta$ up to the kernel mass, while the complement has kernel mass tending to zero. This proves the claim.
\end{proof}

\begin{theorem}[Sharp product-type threshold]\label{thm:nd-product}
For the product-type diagonal family,
\[
K^{(n)}_{\valpha,d}(\mathbf{x})>0\text{ for all }\mathbf{x}\in\Rp^n
\quad\Longleftrightarrow\quad
 d\le \prod_{j=1}^n\alpha_j.
\]
Therefore the constrained minimizer is positive exactly under the same condition.
\end{theorem}

\begin{proof}
Write $A=A_{\valpha}$.  First assume $d\le A$ and put $\nu=A-d\ge0$.  Define
\[
\Phi_n(z)=\sum_{k=0}^\infty\frac{z^k}{(k!)^n},
\qquad z\ge0.
\]
For $s_j\ge0$ and $\beta_j=s_j+\alpha_j$, Tonelli's theorem and the Gamma integral give
\[
\int_{\Rp^n}e^{-\sum_j\beta_jt_j}\Phi_n(\nu t_1\cdots t_n)\,dt
=\frac1{\prod_j\beta_j-\nu}.
\]
Indeed, the series expansion gives a geometric series with ratio $\nu/\prod_j\beta_j<1$; the strict inequality follows from $d>0$.  Taking $s_j=\xi_j^2$ yields a positive Laplace representation for $W_{\valpha,d}^{-1}$.  After the product Gaussian cosine integral,
\[
K^{(n)}_{\valpha,d}(\mathbf{x})=\pi^{-n/2}\int_{\Rp^n}
\frac{\exp\left(-\sum_j\alpha_jt_j-\sum_j x_j^2/(4t_j)\right)}{\sqrt{t_1\cdots t_n}}
\Phi_n(\nu t_1\cdots t_n)\dd\mathbf{t}.
\]
Every factor is positive, so $K^{(n)}_{\valpha,d}>0$.

Now assume $d>A$ and put $\mu=d-A>0$.  If $n=2$, this is Theorem \ref{thm:main}.  Let $n\ge3$.  Fix the first two coordinates and set
\[
\bar A_{12}=\prod_{k=3}^n\alpha_k.
\]
On the two-dimensional frequency face $\xi_3=\cdots=\xi_n=0$,
\[
W_{\valpha,d}(\xi_1,\xi_2,0,\ldots,0)=\bar A_{12}\left(\xi_1^2\xi_2^2+
\alpha_2\xi_1^2+\alpha_1\xi_2^2+\alpha_1\alpha_2+\frac{\mu}{\bar A_{12}}\right).
\]
The bracket is a two-dimensional diagonal symbol with parameters
\[
a=\alpha_2,
\qquad c=\alpha_1,
\qquad d_{12}=\alpha_1\alpha_2+\mu/\bar A_{12}>\alpha_1\alpha_2.
\]
By Theorem \ref{thm:main}, the corresponding two-dimensional face kernel
\[
K^{\mathrm{face}}_{12}(x_1,x_2)=\left(\frac2\pi\right)^2
\int_0^\infty\int_0^\infty
\frac{\cos(x_1\xi_1)\cos(x_2\xi_2)}{W_{\valpha,d}(\xi_1,\xi_2,0,\ldots,0)}\,d\xi_1\,d\xi_2
\]
has a negative value.  Equivalently,
\[
K^{\mathrm{face}}_{12}=\frac1{\bar A_{12}}K^{(2)}_{\alpha_2,\alpha_1,d_{12}}.
\]

Suppose, toward a contradiction, that $K^{(n)}_{\valpha,d}\ge0$ everywhere.  For $\eps>0$ define the Abel average
\[
L_\eps(x_1,x_2)=\int_{\Rp^{n-2}}K^{(n)}_{\valpha,d}(x_1,x_2,z_3,\ldots,z_n)
 e^{-\eps(z_3+\cdots+z_n)}\,dz.
\]
Then $L_\eps\ge0$.  The interchange below is justified as follows: Lemma \ref{lem:nd-reciprocal-integrable} gives $W_{\valpha,d}^{-1}\in L^1(\Rp^n)$, the kernel is bounded by the $L^1$ norm of this reciprocal, and $e^{-\eps(z_3+\cdots+z_n)}$ is integrable in the remaining variables.  Using
\[
\int_0^\infty e^{-\eps z}\cos(z\xi)\,dz=\frac{\eps}{\eps^2+\xi^2}=:P_\eps(\xi),
\qquad \int_0^\infty P_\eps(\xi)\,d\xi=\frac\pi2,
\]
Fubini gives
\[
L_\eps(x_1,x_2)=\left(\frac2\pi\right)^n\int_{\Rp^n}
\frac{\cos(x_1\xi_1)\cos(x_2\xi_2)\prod_{k=3}^nP_\eps(\xi_k)}{W_{\valpha,d}(\boldsymbol{\xi})}\,d\boldsymbol{\xi}.
\]
For fixed $(\xi_1,\xi_2)$, apply Lemma~\ref{lem:half-line-abel} in the variables $(\xi_3,\ldots,\xi_n)$ to the bounded continuous function
\[
\boldsymbol{\zeta}\mapsto \frac1{W_{\valpha,d}(\xi_1,\xi_2,\zeta_3,\ldots,\zeta_n)}.
\]
Since $P_\eps$ has half-line mass $\pi/2$, this gives
\[
\int_{\Rp^{n-2}}\frac{\prod_{k=3}^nP_\eps(\xi_k)}{W_{\valpha,d}(\boldsymbol{\xi})}\,d\xi_3\cdots d\xi_n
\longrightarrow
\left(\frac\pi2\right)^{n-2}
\frac1{W_{\valpha,d}(\xi_1,\xi_2,0,\ldots,0)}.
\]
Moreover
\[
W_{\valpha,d}(\xi_1,\xi_2,\xi_3,\ldots,\xi_n)
\ge W_{\valpha,d}(\xi_1,\xi_2,0,\ldots,0),
\]
so the absolute value of the inner integral is bounded by
\[
\left(\frac\pi2\right)^{n-2}
\frac1{W_{\valpha,d}(\xi_1,\xi_2,0,\ldots,0)}.
\]
The reciprocal of the face symbol is integrable in $(\xi_1,\xi_2)$, so dominated convergence in the face variables gives
\[
\lim_{\eps\downarrow0}L_\eps(x_1,x_2)
=\left(\frac2\pi\right)^n\left(\frac\pi2\right)^{n-2}
\int_0^\infty\int_0^\infty
\frac{\cos(x_1\xi_1)\cos(x_2\xi_2)}{W_{\valpha,d}(\xi_1,\xi_2,0,\ldots,0)}\,d\xi_1\,d\xi_2.
\]
The constant is exactly $(2/\pi)^2$, so the limit is $K^{\mathrm{face}}_{12}(x_1,x_2)$.  Since every $L_\eps$ is nonnegative, the face kernel would be nonnegative everywhere, contradicting the negative face value.  Thus the full kernel changes sign when $d>A$.
\end{proof}

\begin{remark}
The proof shows that, for product-type symbols, the high-dimensional obstruction is already visible on every two-dimensional frequency face.  The supercritical failure is therefore not a new purely high-dimensional phenomenon.
\end{remark}

\subsection{Small non-diagonal perturbations in \texorpdfstring{$n$}{n} dimensions}
Fix finitely many pairs $(S_r,T_r)$ of distinct subsets of $\{1,\ldots,n\}$, $r=1,\ldots,N$.  Define
\[
\begin{aligned}
B_\theta(u,v)=B_0(u,v)&+\sum_{r=1}^N\theta_r\int_{\Rp^n}D_{S_r}uD_{T_r}v\dd\mathbf{x}\\
&+\sum_{r=1}^N\theta_r\int_{\Rp^n}D_{T_r}uD_{S_r}v\dd\mathbf{x},
\end{aligned}
\]
where $B_0$ is the product-type diagonal bilinear form.

\begin{theorem}[Perturbative $n$-dimensional sign stability]\label{thm:nd-stability}
Assume that the product-type diagonal minimizer $u_0$ changes sign. Then there is $\eps_0>0$ such that, whenever $\sum_r|\theta_r|<\eps_0$, the perturbed quadratic form has a unique minimizer under $u(\mathbf{0})=1$. This minimizer changes sign.
\end{theorem}

\begin{proof}
By Lemma \ref{lem:mixed-embedding}, the corner value and every fixed point value are continuous on $E_n$.

The diagonal form is coercive:
\[
B_0(u,u)\ge \mu_0\norm{u}_{E_n}^2,
\qquad
\mu_0=\min\left\{1,d,\prod_{j\notin I}\alpha_j: \emptyset\ne I\subsetneq[n]\right\}>0.
\]
Each symmetric cross term satisfies
\[
\left|\int D_{S_r}uD_{T_r}u\,d\mathbf{x}+\int D_{T_r}uD_{S_r}u\,d\mathbf{x}\right|
\le 2\norm{u}_{E_n}^2.
\]
Hence $B_\theta(u,u)\ge(\mu_0-2\sum_r|\theta_r|)\norm{u}_{E_n}^2$. We may take, for example, $\eps_0=\mu_0/4$.

Let $g_\theta$ be the Riesz representative of the corner value with respect to $B_\theta$.  The same resolvent argument used in Lemma \ref{lem:b-continuity} gives $g_\theta\to g_0$ in $E_n$.  The minimizer is $u_\theta=g_\theta/g_\theta(0)$.  If $u_0(\mathbf{p})<0$ for some point $\mathbf{p}$, then point evaluation gives $u_\theta(\mathbf{p})<0$ for small $\theta$, while $u_\theta(\mathbf{0})=1$.  Thus $u_\theta$ changes sign.
\end{proof}

\begin{remark}
Combining Theorems \ref{thm:nd-product} and \ref{thm:nd-stability}, every product-type supercritical example $d>\prod_j\alpha_j$ remains sign-changing after sufficiently small finite non-diagonal perturbations.  This is a local stability result; it does not claim a global sharp threshold for arbitrary non-diagonal $n$-dimensional forms.
\end{remark}

\section{Capacity criteria for domains and a decaying evolution analogue}\label{sec:long-answers}

Ekeland's three questions are recorded in Long's ICVAM problem list \cite[Section 3.2.2]{Long2008}.  The sharp positivity theorem above addresses the first one in the diagonal family.  This section records two auxiliary results in precisely stated settings.  The first is a Hilbert-space capacity criterion for a free-boundary domain class.  The second concerns the selected decaying branch of a second-order equation, equivalently a first-order nonlocal semigroup.  Neither result is used in the proof of the sharp positivity threshold.

\subsection{A capacity criterion for domains}
Let $\Omega\subset\R^2$ have the origin on its boundary.  Consider
\[
J_\Omega(u)=\int_\Omega\bigl(u_{xy}^2+a u_x^2+2b u_xu_y+
c u_y^2+d u^2\bigr)\dd x\dd y,
\]
where $a>0$, $c>0$, $d>0$, and $ac-b^2>0$.

We use the following free-boundary energy class. For a function defined on $\Omega$, let
\[
\operatorname{supp}_{\overline\Omega}v
=\overline{\{p\in\Omega:v(p)\ne0\}}^{\,\overline\Omega}
\]
be its support relative to $\overline\Omega$. Let $\mathcal D_{\rm fb}(\Omega)$ be the set of all real functions $v$ such that $v\in C^\infty(\Omega)$; the four quantities $v,v_x,v_y,v_{xy}$ belong to $L^2(\Omega)$; $\operatorname{supp}_{\overline\Omega}v$ is compact; and $v$ admits a continuous representative on $\Omega\cup\{0\}$ in some relative neighborhood of the corner. Boundary contact away from $0$ is allowed and no boundary value is imposed on $\partial\Omega\setminus\{0\}$. Contact with $0$ is allowed only through the prescribed continuous representative, and $v(0,0)$ always refers to that representative.

Let $E(\Omega)$ be the Hilbert completion of $\mathcal D_{\rm fb}(\Omega)$ under the norm induced by $J_\Omega^{1/2}$. Define the relative corner capacity
\[
\CornerCap_\Omega(0)=\inf\{J_\Omega(v):v\in \mathcal D_{\rm fb}(\Omega),\ v(0,0)=1\}.
\]
When $\CornerCap_\Omega(0)>0$, the corner value extends continuously to $E(\Omega)$.  When $\CornerCap_\Omega(0)=0$, the condition $u(0,0)=1$ is understood only at the approximating test-function level and is not closed in the energy topology.

\begin{theorem}[Domain capacity criterion]\label{thm:capacity}
The constrained problem on $\Omega$ has a unique variational minimizer in the energy completion if and only if
\[
\CornerCap_\Omega(0)>0.
\]
If $\CornerCap_\Omega(0)=0$, then
\[
\inf\{J_\Omega(v):v\in \mathcal D_{\rm fb}(\Omega),\ v(0,0)=1\}=0.
\]
Moreover no element of the energy completion represents an attained constrained minimizer, because the corner functional does not extend continuously to $E(\Omega)$.
\end{theorem}

\begin{proof}
The lower-order quadratic form is positive definite, so $J_\Omega$ controls
\[
\norm{u}_{L^2}^2+\norm{u_x}_{L^2}^2+\norm{u_y}_{L^2}^2+\norm{u_{xy}}_{L^2}^2.
\]
If $\CornerCap_\Omega(0)>0$, then by scaling
\[
|v(0,0)|^2\CornerCap_\Omega(0)\le J_\Omega(v)
\]
for all test functions.  Thus the corner value extends continuously to the Hilbert completion of $\mathcal D_{\rm fb}(\Omega)$ under the energy norm.  The constraint is a closed affine hyperplane, and the strictly convex quadratic functional has a unique minimizer on it.

If the capacity is zero, choose $v_k\in \mathcal D_{\rm fb}(\Omega)$ with $v_k(0,0)=1$ and $J_\Omega(v_k)\to0$. This proves the normalized test-class infimum is $0$. In this case the estimate $|v(0,0)|^2\le C J_\Omega(v)$ fails, so the corner value does not define a continuous functional on the energy completion. Consequently there is no closed affine constraint $\{u(0,0)=1\}$ inside $E(\Omega)$ to which a direct-method minimizer could belong; an ``attained'' zero-energy limit would be the zero element in $E(\Omega)$ and cannot carry a nonzero corner value.
\end{proof}

\subsection{Cones inside or containing the quadrant}
Let
\[
C_{\theta_1,\theta_2}=\{r(\cos\theta,\sin\theta):r>0,\ \theta_1<\theta<\theta_2\}.
\]

\begin{theorem}[Cones in the positive quadrant]\label{thm:cones}
Assume $0\le\theta_1<\theta_2\le\pi/2$.  If $C_{\theta_1,\theta_2}$ is a proper subcone of the positive quadrant, then
\[
\CornerCap_{C_{\theta_1,\theta_2}}(0)=0.
\]
Thus no variational minimizer exists there.  The full quadrant is exceptional: its corner capacity is positive, and the Ekeland--Nirenberg minimizer exists and is regular.
\end{theorem}

\begin{proof}
Suppose first $\theta_2<\pi/2$.  Then every vertical section has length at most $Mx$.  Let $0<\eps<R$ and define the logarithmic cut-off
\[
F_{\eps,R}(s)=
\begin{cases}
1,&0\le s\le\eps,\\
\log(R/s)/\log(R/\eps),&\eps<s<R,\\
0,&s\ge R.
\end{cases}
\]
Smooth $F_{\eps,R}$ only in the transition intervals, leaving it identically $1$ near $0$ and identically $0$ for $s\ge R$; the following estimates are unchanged up to an absolute constant. Set $v_{\eps,R}(x,y)=F_{\eps,R}(x)$. This function is admissible in the free-boundary class, has $v_{\eps,R}(0,0)=1$, and may touch the non-corner boundary.  Then $v_y=v_{xy}=0$ and
\[
J_C(v_{\eps,R})\le M\int_0^\infty x\bigl(a |F'_{\eps,R}(x)|^2+
d |F_{\eps,R}(x)|^2\bigr)\dd x.
\]
The two terms are bounded by $C/\log(R/\eps)$ and $CR^2$.  Choose $R\downarrow0$ and then $\log(R/\eps)\to\infty$ to get capacity zero.  If $\theta_1>0$, the same argument uses horizontal sections and a cut-off in $y$.

For the full quadrant, Lemma \ref{lem:mixed-embedding} gives the product trace estimate
\[
|v(0,0)|^2\le C\int_{\Rp^2}(v^2+v_x^2+v_y^2+v_{xy}^2)\dd x\dd y\le C'J_{\Rp^2}(v),
\]
so the capacity is positive.
\end{proof}

\begin{proposition}[Quadrant containment and possible loss of regularity]\label{prop:halfplane}
Let $Q=(\Rpo)^2$. Assume that $\Omega$ contains $Q$ with the same corner and that every $v\in \mathcal D_{\rm fb}(\Omega)$ restricts to an element of $D(Q)$ with the same continuous corner representative. Then $\CornerCap_\Omega(0)>0$, so the variational minimizer exists in the free-boundary energy class. However, regularity is not automatic. In the half-plane $H=\R\times\Rp$, for the critical diagonal parameters $b=0$ and $d=ac$, the minimizer is
\[
u_H(x,y)=e^{-\sqrt c |x|-\sqrt a y},
\]
which is not $C^1$ on the interior line $x=0$, $y>0$.
\end{proposition}

\begin{proof}
The pointwise integrand is nonnegative because $ac-b^2>0$: the quadratic form $a p^2+2bpq+cq^2$ is positive definite. Hence, for admissible $v$,
\[
J_Q(v|_Q)\le J_\Omega(v).
\]
The product trace estimate on $Q$ gives
\[
|v(0,0)|^2\le C_Q J_Q(v|_Q)\le C_Q J_\Omega(v),
\]
where the corner value is the same by the restriction hypothesis. Thus $\CornerCap_\Omega(0)>0$.

For the half-plane, split $H$ into the right and left quadrants. Any admissible $u$ has corner value $1$ on both halves, so each half has energy at least the quadrant minimum. The function $e^{-\sqrt c |x|-\sqrt a y}$ attains this lower bound on both halves. It is therefore the minimizer. Its one-sided $x$-derivatives at $x=0$ are $-\sqrt c e^{-\sqrt a y}$ and $+\sqrt c e^{-\sqrt a y}$, so it is not $C^1$ along the interior characteristic line.
\end{proof}

\begin{remark}
The geometric obstruction is characteristic propagation.  Proper cones inside the quadrant have zero corner capacity.  Domains containing the quadrant have a variational solution, but if a coordinate characteristic enters the interior, regularity may fail.
\end{remark}

\subsection{An evolution-equation analogue}
Take $y=t$ as time and consider, for $x\in\R$ and $t>0$,
\[
u_{xxtt}-a u_{xx}-c u_{tt}+d u=0.
\]
After Fourier transform in $x$, the decaying branch satisfies
\[
\partial_t\widehat u(t,\xi)+\lambda(\xi)\widehat u(t,\xi)=0,
\qquad
\lambda(\xi)=\sqrt{\frac{a\xi^2+d}{\xi^2+c}}.
\]
Thus the equation is the first-order nonlocal evolution
\[
\partial_tu+T_{a,c,d}u=0,
\]
where $T_{a,c,d}$ is the Fourier multiplier with symbol $\lambda$.

\begin{theorem}[Decaying branch: well posed but no Sobolev smoothing]\label{thm:evolution}
The evolution above is well posed and exponentially decaying in every Sobolev space $H^s(\R)$.  It has no Sobolev smoothing: for every $t>0$ and $\eps>0$,
\[
u(t)\in H^{s+\eps}(\R)
\quad\Longleftrightarrow\quad
u(0)\in H^{s+\eps}(\R).
\]
In the critical case $d=ac$, one has $\lambda\equiv\sqrt a$.  With initial data $f(x)=e^{-\sqrt c |x|}$,
\[
u(t,x)=e^{-\sqrt c |x|-\sqrt a t},
\]
which remains non-$C^1$ at $x=0$ for every $t>0$.
\end{theorem}

\begin{proof}
The multiplier $\lambda$ is positive and bounded above and below, say $0<m\le\lambda(\xi)\le M<\infty$. Hence the selected semigroup is exponentially bounded above and invertible on its range; in particular
\[
\norm{u(t)}_{H^s}\le e^{-mt}\norm{f}_{H^s},\qquad
\norm{f}_{H^s}\le e^{Mt}\norm{u(t)}_{H^s}.
\]
Since both $e^{-t\lambda(\xi)}$ and its reciprocal are bounded multipliers, membership in $H^{s+\eps}$ is exactly preserved.  The critical example follows from $\lambda\equiv\sqrt a$.
\end{proof}

\begin{remark}
The theorem concerns the selected decaying branch, or equivalently the first-order semigroup generated by $T_{a,c,d}$.  It is not a well-posedness theorem for arbitrary Cauchy data of the full second-order equation, which also contains a growing branch in Fourier space.
\end{remark}

\begin{remark}
If the mixed term $2b u_xu_t$ is retained, the decaying root becomes
\[
\lambda_b(\xi)=\frac{\sqrt{(\xi^2+c)(a\xi^2+d)-b^2\xi^2}+ib\xi}{\xi^2+c}.
\]
Under $ac-b^2>0$ and $d>0$, its real part remains bounded above and below by positive constants.  The same well-posedness and no-smoothing conclusion holds; the diagonal case only avoids a harmless complex phase.
\end{remark}

\section[Further questions and applications]{Further questions and applications}

The results in this section are separated from the main theorem. They record consequences, diagnostics, and open problems suggested by the proof. The steepest-descent statement below is a standard consequence of a one-dimensional saddle-point theorem; it is not needed for Theorem~\ref{thm:main}.

The fist question is about the geometry of the supercritical negative set under the setting of Theorem~\ref{thm:main}.
For $m>1$, let
\[
 N_m=\{(x,y)\in\Rp^2:K_m(x,y)<0\}.
\]
\begin{question}
How many unbounded negative components does $N_m$ have? Is $N_m$ connected after a compact core is added? How does it move as $m$ increases?
\end{question}

The branch-cut formula gives a precise asymptotic answer near infinity. Put
\[
 A_m=\sqrt{\frac{m-1}{2}},
 \qquad
 H(\alpha)=\int_0^\infty e^{-u}u^{-1/2}\cos\left(\frac\alpha{\sqrt u}\right)\dd u.
\]
Proposition~\ref{prop:boundary} gives the locally uniform limit
\[
 e^x\sqrt x\,K_m\left(x,\frac{\rho}{\sqrt x}\right)
 \longrightarrow \frac2{\pi\sqrt{2(m-1)}}H(A_m\rho).
\]
Therefore compact subintervals of $\{H<0\}$ give negative boundary-layer tongues, while compact subintervals of $\{H>0\}$ give positive corridors.

The following standard saddle-point expansion, used here only as a diagnostic for open problems, shows that $H$ has infinitely many sign intervals. Writing
\[
 I(\alpha)=\int_0^\infty e^{-u}u^{-1/2}e^{i\alpha/\sqrt u}\dd u,
 \qquad H(\alpha)=\operatorname{Re}I(\alpha),
\]
and scaling $u=\alpha^{2/3}v$, the phase has the saddle $v_0=2^{-2/3}e^{-i\pi/3}$ connected with the positive real axis. The usual steepest-descent theorem \cite[Ch.~7]{Olver1974} gives
\[
 H(\alpha)=\frac{2\sqrt\pi}{\sqrt3}e^{-\sigma\alpha^{2/3}}
 \left(\cos(\tau\alpha^{2/3})+O(\alpha^{-2/3})\right),
 \qquad \alpha\to\infty,
\]
with
\[
 \sigma=\frac3{2^{5/3}},
 \qquad
 \tau=\frac{3\sqrt3}{2^{5/3}}.
\]
The verification required by the theorem is the standard one: the positive real ray can be deformed to the steepest path through $v_0$, the endpoint contribution at $0$ is exponentially small on the deformed contour, and the tail is dominated by the negative real part of the phase. This asymptotic explains the infinitely many boundary-layer tongues but leaves their global topology open.
It is also interesting to study the cost of forbidding signed localizers.
Let
\[
 \mathcal A=\{u\in E_2:u(0,0)=1\},
 \qquad
 \mathcal A_+=\{u\in\mathcal A:u\ge0\hbox{ a.e.}\},
\]
and define
\[
 E_*:=\min_{u\in\mathcal A}J_{a,c,d}(u),
 \qquad
 E_+:=\inf_{u\in\mathcal A_+}J_{a,c,d}(u).
\]
If $d\le ac$, then $u_*>0$ and $E_+=E_*$. If $d>ac$, then $u_*$ changes sign and, for every $v\in\mathcal A_+$,
\begin{equation}\label{eq:variance-gap}
 J_{a,c,d}(v)-J_{a,c,d}(u_*)=J_{a,c,d}(v-u_*)
 \ge d\norm{u_*^-}_{L^2(\Rp^2)}^2>0.
\end{equation}
Taking the infimum over $v\in\mathcal A_+$ gives $E_+-E_*\ge d\norm{u_*^-}_{L^2(\Rp^2)}^2$. The equality in \eqref{eq:variance-gap} follows from the Euler orthogonality $B_{a,c,d}(u_*,h)=0$ for all $h(0,0)=0$. In finite-dimensional Monte Carlo discretizations, the same identity says that negative weights are control weights, not probabilities.

The diagonal theory suggests three separate issues for a full non-diagonal theorem: coercivity of the quadratic form, positivity of a transform representation, and a sign-forcing mechanism when transform positivity fails. Corollary \ref{cor:intro-b} and Theorem~\ref{thm:nd-stability} give only local stability near sign-changing diagonal examples. A structural criterion for arbitrary non-diagonal forms remains open.

\begin{question}
What is the sharp criterion beyond the diagonal cases?
\end{question}

A plausible target is a positive transform representation for the reciprocal symbol, or, in genuinely coupled sine--cosine coordinates, a matrix-valued complete monotonicity condition for the inverse symbol. If such positivity fails on a lower-dimensional face, the branch-cut method of Proposition~\ref{prop:branch} suggests a route to sign change. Turning this scheme into an explicit theorem is a separate problem.

\small {\bf \small  Declarations of interest:}
 none.

{\bf \small Data availability statement:} There are no new data associated with this article.

{\bf \small AI assistance statement:}
The authors used OpenAI models to assist with proof exploration, symbolic checking, numerical experimentation, and manuscript editing. The mathematical validation, the final proof choices, and the final text are the responsibility of the human authors.

 \newpage
\noindent {Qi Guo\\
School of Mathematics,\\
Renmin University of China, Beijing, 100872, P.R. China\\
e-mail: qguo@ruc.edu.cn}
\medskip
\\
\noindent {Xueping Huang\\
school of Mathematics and Statistics,\\
 Nanjing University of Information Science and
Technology, Nanjing 210044, P. R. China\\
email: hxp@nuist.edu.cn}
\medskip
\\
\noindent {Yi C. Huang\\
  School of Mathematical Sciences,\\
Nanjing Normal University, Nanjing 210023, P.R. China \\ 
e-mail: Yi.Huang.Analysis@gmail.com}

\end{document}